\pdfoutput=1
\RequirePackage{ifpdf}
\ifpdf 
\documentclass[pdftex]{sigma}
\else
\documentclass{sigma}
\fi

\usepackage{mathabx}
\usepackage{mathtools}
\usepackage{tikz}
\usetikzlibrary{shapes,arrows,calc,positioning}

\newcommand{\mcN}{\mathcal{N}}

\newcommand{\mcS}{\mathcal{S}}

\newcommand{\C}{\mathbb{C}}

\newcommand{\ZZ}{\mathbb{Z}}

\newcommand{\isom}{\cong}

\newcommand{\ot}{\otimes}

\newcommand{\GL}{\operatorname{GL}}

\newcommand{\SL}{\operatorname{SL}}
\newcommand{\id}{\operatorname{id}}

\renewcommand{\Im}{\operatorname{Im}}

\newcommand{\tn}{\textnormal}

\newcommand{\End}{\textnormal{End}}
\newcommand{\ott}{\bigotimes}
\newcommand{\B}{\big}

\newcommand{\op}{\oplus}
\newcommand{\bop}{\bigoplus}
\newcommand{\Tr}{\textnormal{Tr}}

\newcommand{\actson}{\curvearrowright}
\newcommand{\Ud}{\textnormal{$\textnormal{U}_{\textbf{d}}$}}
\newcommand{\Sl}[1]{\textnormal{$\SL_{\textbf{#1}}$}}
\newcommand{\GLd}{\textnormal{$\GL_{\textbf{d}}$}}
\newcommand{\SLd}{\textnormal{$\SL_{\textbf{d}}$}}
\newcommand{\U}{\textnormal{$\textnormal{U}$}}

\newcommand{\bfd}{\mathbf{d}}

\numberwithin{equation}{section}

\newtheorem{Theorem}{Theorem}[section]
\newtheorem{Corollary}[Theorem]{Corollary}
\newtheorem{Lemma}[Theorem]{Lemma}
\newtheorem{Proposition}[Theorem]{Proposition}
 { \theoremstyle{definition}
\newtheorem{Definition}[Theorem]{Definition}
\newtheorem{Example}[Theorem]{Example}
\newtheorem{Remark}[Theorem]{Remark} }

\begin{document}


\newcommand{\arXivNumber}{1507.03350}

\renewcommand{\PaperNumber}{028}

\FirstPageHeading

\ShortArticleName{A Complete Set of Invariants for LU-Equivalence of Density Operators}

\ArticleName{A Complete Set of Invariants for LU-Equivalence\\ of Density Operators}

\Author{Jacob TURNER~$^\dag$ and Jason MORTON~$^\ddag$}

\AuthorNameForHeading{J.~Turner and J.~Morton}

\Address{$^\dag$~Korteweg-de Vries Institute, University of Amsterdam, 1098 XG Amsterdam, The Netherlands}
\EmailD{\href{mailto:jacob.turner870@gmail.com}{jacob.turner870@gmail.com}}
\URLaddressD{\url{http://www.jacobwadeturner.weebly.com}}

\Address{$^\ddag$~Department of Mathematics, The Pennsylvania State University,\\
\hphantom{$^\ddag$}~University Park, PA 16802, USA}
\EmailD{\href{mailto:morton@math.psu.edu}{morton@math.psu.edu}}

\ArticleDates{Received November 26, 2016, in f\/inal form April 28, 2017; Published online May 02, 2017}

\Abstract{We show that two density operators of mixed quantum states are in the same local unitary orbit if and only if they agree on polynomial invariants in a certain Noetherian ring for which degree bounds are known in the literature. This implicitly gives a f\/inite complete set of invariants for local unitary equivalence. This is done by showing that local unitary equivalence of density operators is equivalent to local ${\rm GL}$ equivalence and then using techniques from algebraic geometry and geometric invariant theory. We also classify the SLOCC polynomial invariants and give a degree bound for generators of the invariant ring in the case of $n$-qubit pure states. Of course it is well known that polynomial invariants are not a complete set of invariants for SLOCC.}


\Classification{20G05; 20G45; 81R05; 20C35; 22E70}

\section{Introduction}

Consider the \emph{local unitary} group $\Ud:=\times_{i=1}^n{\U\big(\C^{d_i}\big)}$, a product of unitary groups where $d=(d_1, \dots, d_n)$ are positive integer dimensions. Let $V_i$ be a $d_i$-dimensional complex Hilbert space and $V = \otimes_{i=1}^n V_i$. Then $\Ud$ acts on the vector space $\End(V)=\ott_{i=1}^n\End(V_i)$, $\dim (V_i)=d_i$, by linear extension of the action
\begin{gather*}
\times_{i=1}^n{g_i}.\bigg(\ott_{i=1}^n{M_{i}}\bigg):=\ott_{i=1}^n{g_iM_{i}g_i^{-1}}.
\end{gather*} This in turn can be naturally extended to an action on $\End(V)^{\op m}$ by simultaneous conjugation.

This action on density operators is important for understanding entanglement of quantum states \cite{biamonte2013tensor,gour2013classification, grassl1998computing,hero2009stable,johansson2012topological,kraus2010local,MR2039690,luque2006algebraic,macicazek2013many}. Many of the most important notions of entanglement are invariant under the action of $\Ud:=\times_{i=1}^n{\U\big(\C^{d_i}\big)}$ \cite{PhysRevA.62.062314,luque2007unitary}. Entanglement in turn relates to quantum computation~\cite{NielsenChuang,raussendorf2001one}, quantum error correction~\cite{NielsenChuang}, and quantum simulation~\cite{lloyd1996universal}. Two density operators in the same $\Ud$ orbit are said to be local unitary (LU)-equivalent.

 When considering the local unitary equivalence of two mixed quantum states, one can either take two views: the f\/irst is that the entire system as a whole is related by a local unitary change of basis. In this case we look at a single density operator acted on by $\Ud$. The second is that by considering the same change of basis on each pure state in the mixture, one can take one mixed system to the other. In the latter case, we are looking at local unitary group acting in a simultaneous fashion on the $m$ pure states in the mixed state. Furthermore, our proofs are simplif\/ied by considering the problem of classifying the invariants of $\End(V)^{\op m}$ for all $m$ simultaneously.

In this paper, we concern ourselves with the problem of f\/inding a \emph{complete set of invariants} for density operators. By this we mean a set of $\Ud$-invariant functions $f_1,\dots,f_s$ such that two density operators $\Psi_1$ and $\Psi_2$ are in the same $\Ud$ orbit if and only if $f_i(\Psi_1)=f_i(\Psi_2)$ for all $i$. In the f\/irst part of this paper, we will restrict our attention to polynomial invariants of this action.

\begin{Remark}\label{rem:caveat}
As a caveat: throughout this paper, when we say polynomial invariants, we mean those invariants that are polynomials in the ring $\C[v_1,\dots,v_n]$ where the $v_i$ are a basis for space $\End(V)$ viewed as a \emph{complex vector space}. Quite frequently in the physics literature, the term polynomial invariant refers to polynomials in the basis of $\End(V)$ as a real vector space. This allows for invariants such as the Hermitian form. It is known that the set of all polynomial invariants found by viewing $\End(V)$ as a real vector space is complete~\cite{onishchik2012lie}. It is an interesting consequence of our main theorem, however, that this larger set of polynomial invariants is not necessary for f\/inding a complete set of invariants, which is important if we wish to f\/ind minimal complete sets of invariants.
\end{Remark}

We denote the ring of invariants for $G\actson V$, $V$ a vector space over a f\/ield $k$, by $k[V]^G$. We recall that $k[V]$ is to be interpreted as the polynomial ring $k[v_1,\dots,v_n]$ where $v_1,\dots,v_n$ form a basis for $V$. This paper focuses on the completeness of these invariants; f\/initeness results have been found previously by exhibiting degree bounds on generators and we do not make further contributions in this regard. We show that for density operators in $\End(V)$, polynomial invariants of degree at most
\begin{gather*}
\max \left\{2,\frac{3}{8}\max\{d_i\}m^2\dim(V)^4(2n)^{2\delta}\right\},
\end{gather*} where $\delta=\sum\limits_{i=1}^m{(d_i-1)}$ distinguish their orbits (Corollary~\ref{cor:maincor}).

Throughout this paper, whenever possible, our theorems hold for the invariant ring \linebreak $k[\End(V)]^{\GLd}$, where $k$ is an algebraically closed f\/ield of characteristic zero which has a Hilbert space structure. Otherwise, $k=\C$. We wish to f\/ind a f\/inite (and preferably small) generating set of invariants. We consider the constant
\begin{gather*}\beta_G(V):=\min \big\{d\,|\, k[V]^G\tn{ is generated by polynomials of degree}\le d\big\}.\end{gather*}
Upper bounds for this constant have been studied in previous works. We discuss the specif\/ic upper bounds for $\beta_{\Ud}(\End(V)^{\op m})$ that arise from general bounds given in the literature, thus giving a f\/inite set of invariants that we show is complete.

We now give a brief example to show why completeness of invariants is a non-trivial phenomenon requiring proof. Indeed, it is far from obvious that one cannot f\/ind two density ope\-ra\-tors that are not in the same local unitary orbit but take the same value for every polynomial invariant evaluated on them.

\begin{Example}
 Consider $\C^2$ being acted upon by the group $\C^\times$ in the following manner: $\lambda.(x,y):=(\lambda x,\lambda^{-1} y)$. It is clear that the only invariant is $xy$. However, if $xy=0$, then there are three distinct orbits that $(x,y)$ could be in: $X:=\{(x,0)\,|\, x\in \C\setminus\{0\}\}$, $Y:=\{(0,y)\,|\, y\in\C\setminus\{0\}\}$, or the origin. So we say that these three orbits, while distinct, cannot be separated (or distinguished) by invariants. This problem can be seen in this example in the following way: most orbits are hyperbolas def\/ined by $xy=c$ for $c\ne 0$. Therefore each of these orbits is a Euclidean closed subset.

 However, for the three problematic orbits, two of them are not closed and contain the origin in their closure. As such, given any continuous function constant on $Y$, it is also constant on the whole $y$-axis. Similarly for the functions constant on $X$. Given any continuous function that is constant on orbits, we see that it must take the same value on $X$ and $Y$ since it is constant on the entire $x$-axis and constant on the entire $y$-axis and these two sets intersect.

 The goal of this paper is to show that such a phenomenon does not occur if we restrict our attention to density matrices under the local unitary action.
\end{Example}

The above example contained orbits that could not be distinguished even by all \emph{continuous invariants} (as opposed to just the polynomial invariants) and thus we could use the Euclidean topology to understand the problem. However, since we are interested in polynomial invariants, the more natural topology is the Zariski topology. We wish to show that the Zariski closure of two $\Ud$ orbits of two inequivalent density operators do not intersect. Throughout the paper, we will assume that we are working in the Zariski topology. When we say the closure of a set $X$, which we will denote $\overline{X}$, we will mean the Zariski closure.

We remind the reader that the Zariski closure of a set $X$ is the largest set $\overline{X}$, containing $X$, such that every polynomial that vanishes identically on $X$ must also vanish identically on $\overline{X}$. If $X=\overline{X}$, we say that $X$ is \emph{Zariski closed}. We call $X$ \emph{Zariski dense} in $Y$ if every polynomial that vanishes identically on $X$ must vanish identically on~$Y$.

We wish to use techniques from classical invariant theory and algebraic geometry. The group~$\Ud$ does not satisfy the necessary conditions for the theorems we wish to use (it is not reductive). So instead, we consider the group $\GLd:=\times_{i=1}^n{\GL(\C^{d_i})}$, which is reductive (over $\C$, this means that all of its rational representations are semi-simple). We shall see that for this group action, the Zariski closure of the orbits will actually coincide with its Euclidean closure. This simplif\/ies the problem greatly. We note that throughout the paper, a $\GLd$ orbit or set is not assumed to be closed unless explicitly stated.

We say that a group $G$ acts on a vector $V$ rationally, or equivalently, is a rational representation if the map $G\to \End(V)$ is given in every coordinate by a rational function that is well-def\/ined everywhere on $G$. The following two propositions tell us that studying $\GLd$ is suf\/f\/icient. Rational functions are continuous maps \emph{with respect to the Zariski topology} and so send Zariski dense subsets to Zariski dense subsets.

\begin{Proposition}\label{prop:denseinvariants}
 If $H$ is a Zariski dense subgroup of $G$ and $\rho$ is a rational representation of~$G$ acting on a vector space~$V$, $k[V]^G=k[V]^H$.
\end{Proposition}
\begin{proof}
 The representation $\rho$ is a continuous map from $G\to \GL(V)$ with respect to the Zariski topology by assumption of the rationality of the representation. For every $v\in V$, consider the map $\varphi_v\colon G\to G.v$ given by~$g\mapsto g.v$. This is also a continuous map and it implies that for every $v\in V$, $H.v$ is dense in $G.v$ since the continuous image of dense sets are dense. The invariant ring is the ring of polynomials which are constant on orbit closures. Since the orbit closures of~$H$ and $G$ coincide, their invariant rings must be the same.
\end{proof}

It is well known that $\U(\C^{d_i})$ is a Zariski dense subgroup of $\GL(\C^{d_i})$, a fact sometimes known as Weyl's trick. This implies that $\Ud$ is Zariski dense in $\GLd$, so $\C[\End(V)^{\op m}]^{\Ud}=\C[\End(V)^{\op m}]^{\GLd}$. Furthermore, the action $\GLd\actson\End(V)^{\op m}$ is not faithful since conjugating a matrix $M$ by $\alpha I$ for $\alpha\in \C$ leaves $M$ f\/ixed. Therefore, we have that $\C[\End(V)^{\op m}]^{\tn{SU}_{\textbf{d}}}=\C[\End(V)^{\op m}]^{\SLd}=\C[\End(V)^{\op m}]^{\GLd}$.

\begin{Proposition}\label{prop:glimpliesu}
Two Hermitian matrices are in the same $\GLd$ orbit if and only if they are in the same $\Ud$ orbit.
\end{Proposition}
\begin{proof}
 Consider the polar decomposition of $\ot_{i=1}^n{g_i}=(\ot_{i=1}^n{p_i})(\ot_{i=1}^n{u_i})$ where the $p_i$ are invertible Hermitian matrices and the $u_i$ are unitary. We can assume without loss of generality that all $u_i=\id$ since it does not change the $\Ud$ orbit we are in. So note that $P=\ot_{i=1}^n{p_i}$ is a Hermitian matrix. Let $H$ be Hermitian and suppose that $PHP^{-1}$ is Hermitian. Then $PHP^{-1}=(PHP^{-1})^\dagger=P^{-1}HP$, implying that $P^2HP^{-2}=H$. This implies that either $P$ commutes with $H$, and thus $PHP^{-1}$ is in the same $\Ud$ orbit as~$H$, or $P^2=PP^\dagger=\id$, implying that $P$ was unitary.
\end{proof}

By restricting the invariant functions we study to be polynomials, Propositions~\ref{prop:denseinvariants} and~\ref{prop:glimpliesu} tell us that we can focus our attention instead on the ring $\C[\End(V)]^{\GLd}$. However, we may run into the problem that two density operators are in distinct $\GLd$ orbits but cannot be distinguished by invariant polynomials. We show in Section~\ref{sec:closedorbits} that~$\GLd$ orbits of density operators can always be separated by invariant polynomials.

\subsection{Background}
Previous work on LU-equivalence includes both the invariant theory and normal form approaches. Invariants for LU-equivalence are studied in \cite{grassl1998computing} and much work has been done to understand the invariant rings especially in the case $V_i\isom \C^2$ \cite{PhysRevA.88.042304,PhysRevA.86.010303,ziman2001local}.

Many polynomial invariants (as well as other invariants) have been identif\/ied for this group action. In fact, all polynomial invariants have been found, however this fact has not been proven. We do so in this paper. Invariant based approaches are sometimes criticized because of the dif\/f\/iculty of interpreting the invariants \cite{li2013classification, williamson2011geometric}.

A necessary and suf\/f\/icient condition for LU-equivalence of a generic class of multipartite pure qubit states is given by Kraus in \cite{kraus2010local} using a normal form. In~\cite{zhang2013local} the non-degenerate mixed qudit case is covered. Finally a necessary and suf\/f\/icient condition for LU-equivalence of multipartite mixed states, including degenerate cases, is given by Zhang et al.~in~\cite{PhysRevA.88.042304}, also based on a normal form. A similar normal form is given in \cite{li2013classification, liu2012local} based on HOSVD. The mixed case is treated by purif\/ication, so $\rho \sim \rho$ if and only if $\Psi_\rho \sim \Psi_\rho$.

The normal form approaches work by locally diagonalizing the density operator. They require that the coef\/f\/icients of the pure or mixed states be known precisely and explicitly so that the normal forms may be computed. However, given two quantum states in the laboratory, determining the density operators $\Psi_1$ and $\Psi_2$ is not necessarily feasible.

Nevertheless, computing the values of invariant polynomials for a density operator may not require such knowledge. Given a \emph{bipartition} $A\!:\!B$ of $V$, where $A$ and $B$ are complementary subsystems, and a density operator $\rho$, we then note the following equality
\begin{gather*}
\Tr(\Tr_A(\rho)^q)=\exp \big((1-q)H^{AB}_q(\rho)\big),
\end{gather*} which is a polynomial for $q$ a natural number. The R\'enyi entropies \cite{baez2011renyi,biamonte2013tensor,biamonte2015tensor,eisert2010colloquium, Renyi2} are a~well-studied measurement of entanglement. Positive integral ($q \in \ZZ_{\geq 1}$) R\'enyi entropies can be measured experimentally without computing the density operators explicitly \cite{abanin2012measuring,PhysRevLett.106.150404,daley2012measuring,pichler2013thermal,schachenmayer2013entanglement}. This suggests that it may be possible to compute the value of~$\Psi_1$ on an invariant without computing~$\Psi_1$. This would mean that the invariant polynomials can be expressed as a series of measurements that can be carried out on a quantum state in the laboratory. However, whether or not this is true is still unresolved.

\subsection{Organization of the paper}

In Section \ref{sec:prelim}, we cover the preliminaries of invariant theory we shall need. In Section \ref{sec:invariants}, we classify the invariants of $\GLd$ acting $\End(V)^{\op m}$; Theorem \ref{thm:genend} gives the result. In Section \ref{sec:closedorbits} we prove the title result. Theorem~\ref{thm:normal} and Corollary \ref{cor:densitysep} show that density operators can be distinguished by polynomial invariants. We then draw on results from dif\/ferent sources to f\/ind f\/inite sets of polynomial invariants that are complete. Lastly, in Section~\ref{sec:slocc}, we discuss a related problem in the study of quantum entanglement. Given the group $\SLd:=\times_{i=1}^n{\SL(\C^{d_i})}$, there is an action on~$V$ by $(g_1,\dots,g_n).v:=(\ot_{i=1}^n{g_i})v$. There has been much research done on computing invariants of this action, known as SLOCC. An algorithm was given that computes all such invariants \cite{gour2013classification}. For small numbers of qubits (up to four), f\/inite generating sets are
explicitly known \cite{parfenov2001orbits,verstraete2002four} (although there was a misprint in~\cite{verstraete2002four} that was corrected in \cite{chterental2007normal}). Work has been done for higher numbers of qubits \cite{grassl1998computing,
hero2009stable,luque2006algebraic}. In Theorem \ref{thm:slocc}, we classify all invariants for this action for any number of qubits.

\section{Preliminaries}\label{sec:prelim}

In this section, we state the necessary def\/initions and theorems we shall need for the rest of this paper.

\begin{Definition}\label{def:multihom}
 A function $f\in k[V_1\op\cdots\op V_r]$ is \emph{multihomogeneous} of degree $t=(t_1,\dots,t_r)$ if $f(\lambda_1 v_1,\dots,\lambda_r v_r)=\lambda_1^{t_1}\cdots\lambda_r^{t_r}f(v_1,\dots,v_r)$.
\end{Definition}
\begin{Definition}
Suppose $f\in k\big[V_1^{\op t_1}\op\cdots\op V_r^{\op t_r}\big]$ is a multilinear polynomial. Then the restitution of $f$, $\mathcal{R}f \in k[V_1 \op \cdots \op V_r]$ is def\/ined by
\begin{gather*}
\mathcal{R}f(v_1,\dots,v_r)=f(\underbracket{v_1,\dots,v_1}_{\text{$t_1$}},\dots,\underbracket{v_r,\dots,v_r}_{\text{$t_r$}}).
\end{gather*} The result is a multihomogeneous function.
\end{Definition}

The notion of restitution simply makes formal the idea that if one is given a multilinear function $f(X_1,\dots,X_m)$, then one may force some of the variables to be equal and the resulting function is no longer multilinear. For example, the function $\Tr(XY^2)$ is not multilinear in the variables $X$ and $Y$. However, it may be seen as the multilinear function $\Tr(XYZ)$ where we have imposed the restriction that $Y=Z$. Thus $\Tr(XY^2)$ is a multihomogeneous function that is a restitution of the multilinear function $\Tr(XYZ)$.

By taking restitutions of multilinear invariants, we can recover generators for the ring of all invariants. An important observation that we shall use later is that if two representations have the same multilinear invariants, then their invariant rings coincide.

Invariant rings can always be generated by multihomogeneous polynomials. The reason for this is that the action of a linear group does not change the degree of the polynomials since it only involves a linear change of variables.

\begin{Proposition}[\cite{kraft2000classical}]\label{prop:rest}
 Let $V_1,\dots,V_m$ be representations of a group $G$. Then every multihomogeneous invariant $f\in k[V_1\oplus\cdots\oplus V_m]^G$ of degree $t=(t_1,\dots,t_m)$ is the restitution of a~multilinear invariant $F\in k\big[V_1^{\op t_1}\oplus\cdots\oplus V_m^{\op t_m}\big]^G$.
\end{Proposition}

So while it is not true that every invariant is the restitution of a multilinear invariant, the restitutions of multilinear invariants will generate the invariant ring. Furthermore, this ring is f\/initely generated for certain kinds of groups.

\begin{Theorem}[\cite{hilbert1890theorie,hilb:93}]\label{thm:noetherian}
 If $W$ is a $G$-module and the induced action on $k[W]$ is completely reducible, the invariant ring $k[V]^G$ is finitely generated.
\end{Theorem}

So we know by the above Theorems that $k[\End(V)^{\op m}]^{\GLd}$ is always f\/initely generated.

\begin{Definition}
 The \emph{null cone} of an action $G\actson V$ is the set vectors $v$ such that $0\in\overline{G.v}$. We denote it by $\mcN_V$. Equivalently, $\mcN_V$ are those $v\in V$ such that $f(v)=f(0)$ for all invariant polynomials $f$.
\end{Definition}

When studying orbit closures, the following theorem is a powerful tools when dealing with reductive groups. It gives a picture of which orbits cannot be distinguished from each other by means of polynomial invariants.
\begin{Theorem}[\cite{brion2008representations,MR1304906}]\label{thm:orbitclasses}
 Given an action of an algebraic group $G\actson V$, the orbit closure $\overline{G.x}$ is the union of $G.x$ and orbits of strictly smaller dimension. An orbit of minimal dimension is closed, thus every closure $\overline{G.x}$ contains a closed orbit. Furthermore, this closed orbit is unique.
\end{Theorem}

The following theorem gives us a way to reason about points in the orbit closure of a reductive group action that are not in the orbit. Indeed, as it turns out, all such boundary points can be found as endpoints of a path inside of the orbit. This, combined with the fact that every Zariski closed set is Euclidean closed, implies that for reductive group actions, the Zariski closure and Euclidean closure of an orbit coincide.

\begin{Theorem}[the Hilbert--Mumford criterion \cite{kempf1978instability}]\label{thm:hbcriterion}
 For a linearly reductive group $G$ acting on a variety $V$, if $\overline{G.w}\setminus G.w\ne\varnothing$, then there exists a $v\in\overline{G.w}\setminus G.w$ and a $1$-parameter subgroup $($or cocharacter$)$ $\lambda\colon k^\times\to G$ $($where $\lambda$ is a homomorphism of algebraic groups$)$, such that $\lim\limits_{t\to 0}{\lambda(t).w}=v$.
\end{Theorem}

Note that for the action of $\GLd\actson \End(V)$, if $G.w$ is not closed, then for any $v\in\overline{G.w}\setminus G.w$, there is a cocharacter $\lambda(t)$ such that $\lim\limits_{t\to 0}{\lambda(t)w\lambda(t)^{-1}}=v$. Indeed, we know that if $\overline{G.w}\setminus G.w\ne\varnothing$, there is some $v'$ and cocharacter $\mu(t)$ such that $\lim\limits_{t\to 0}{\mu(t)w\mu(t)^{-1}}=v'=gvg^{-1}$ for some $g\in \GLd$. Then note that if we def\/ine $\lambda(t)=g^{-1}\mu(t)g$, we get a cocharacter of $\GLd$ sending $w$ to $v$ as desired.

So we have that every orbit class has a unique representative given by a closed orbit and every closed orbit trivially lies in some orbit class. This motivates the def\/inition of dif\/ferent types of points in $V$ with respect to an action of $G$.

\begin{Definition}
Given an action $G\actson V$ and a point $v\in V\setminus\{0\}$, then $v$ is called
\begin{enumerate}\itemsep=0pt
 \item[(a)] an \emph{unstable point} if $0\in\overline{G.v}$,
\item[(b)] a \emph{semistable point} if $0\notin\overline{G.v}$,
\item[(c)] a \emph{polystable point} if $G.v$ is closed,
\item[(d)] or a \emph{stable point} if $G.v$ is closed and the stabilizer of $v$ is f\/inite.
\end{enumerate}
\end{Definition}

These def\/initions have been reinterpreted in terms of the study of entanglement of pure states by Klyachko \cite{klyachko2002coherent}. For example, every stable point is in the orbit of a completely entangled state and entangled states are simply the semistable points.

Given an action of a reductive group $G\actson V$, there is a way to write every vector that highlights whether or not its orbit is closed and a representative in the closed orbit its orbit closure contains.

\begin{Definition}\label{def:jordandecomp}
 Given an action $G\actson V$, a \emph{Jordan decomposition} of a point $v$ is given by $v=v_s+v_n$ where $v_s$ is a polystable point and $v_n$ is an unstable point.
\end{Definition}

For a rational representation of a reductive group $G\actson V$, such a Jordan decomposition always exists, although it is not unique. This is well known (cf.~\cite{le1990semisimple}), but we include a proof for completeness.
\begin{Theorem}
 For a reductive group action $\varphi\colon G\to\GL(V)$ a Jordan decomposition always exists.
\end{Theorem}
\begin{proof}
By Theorem \ref{thm:orbitclasses}, $\overline{\varphi(G)v}$ contains a polystable point $v_s$, and by the Hilbert--Mumford criterion (Theorem~\ref{thm:hbcriterion}), there exists a cocharacter $\lambda(t)\colon k^{\times}\to G$ such that $\lim\limits_{t\to 0}{\varphi(\lambda(t))v}$ is polystable. Since $\varphi(\lambda(t))$ is diagonalizable, there is some $g\in\GL(V)$ such that $\lim\limits_{t\to 0}{g\varphi(\lambda(t))g^{-1}gv}$ $=gv_s$ for some $v_s\in V$.

Now if $g\varphi(\lambda(t))g^{-1}$ is diagonal, then $g\varphi(\lambda(t))v$ is the vector $gv$ with every entry multiplied by a some non-negative power of $t$ (since the limit exists). The unstable part of~$gv$, denoted~$gv_n$, is the all zero vector except for those entries of $gv$ that get multiplied by a positive power of~$t$. The stable part is $gv_s=gv-gv_n$. Then we see that $\lim\limits_{t\to 0}{g\varphi(\lambda(t))g^{-1}gv_s}=gv_s$ and so $\lim\limits_{t\to 0}{\varphi(\lambda(t))v_s}=v_s$. Then we let $v_n=v-v_s$. We quickly see that $\lim\limits_{t\to0}{\varphi(\lambda(t))v}=v_s$ and thus $\lim\limits_{t\to0}{\varphi(\lambda(t))v_n}=0$. Then $v=v_s+v_n$ is the Jordan decomposition.
\end{proof}

\section[Describing the ring $\protect{k[\End(V)^{\op m}]^{\GLd}}$]{Describing the ring $\boldsymbol{k[{\rm End}(V)^{\op m}]^{\GLd}}$}\label{sec:invariants}

In this section, we describe the invariant ring $k[\End(V)^{\op m}]^{\GLd}$ by giving a description of all multihomogeneous elements of said ring. We follow Kraft and Procesi's (specif\/ically Chapter~4 in~\cite{kraft2000classical}) treatment of the fundamental theorems, generalizing to local conjugation by $\GLd$; see also Leron~\cite{leron1976trace}.

Let us consider the representation of $\GLd$ given by $\mu\colon \GLd=\times_{i=1}^n \GL\big(k^{d_i}\big)\to\End (V^{\ot m} )$ def\/ined by
\begin{gather*}\mu(g_1,\dots,g_n)\ott_{i=1}^n{\ott_{j=1}^m{v_{ij}}}:=\ott_{i=1}^n{\ott_{j=1}^m{g_iv_{ij}}}\end{gather*}
extended linearly. Let $\mcS_m^n$ be the $n$-fold product of the symmetric group of order $m$. The $\GLd$ action commutes with the representation of $\rho\colon \mcS_m^n\to\End(V^{\ot m})$ def\/ined by
\begin{gather*}\rho(\sigma_1,\dots,\sigma_n)\ott_{i=1}^n{\ott_{j=1}^m{v_{ij}}}:=\ott_{i=1}^n{\ott_{j=1}^m{v_{i\sigma^{-1}_i(j)}}}\end{gather*}
extended linearly. We will show that the centralizer of this action of~$\GLd$ is precisely the described action of $\mcS_m^n$. In the case of $n=1$, the group algebra of $\mcS_m$ is precisely the centralizer of~$\GL(V)$ acting on this space. Furthermore, over an algebraically closed f\/ield, the centralizer of the centralizer of an algebra is the original algebra. This a classical theorem called the \textit{double centralizer theorem} (cf.~\cite{landsberg2012tensors}).

Given a representation $\varphi\colon G\to\End(V^{\ot m})$, denote by $\langle G\rangle_\varphi$ the linear span of the image of $G$ under the map $\varphi$. We denote the centralizer of the image of $\mu$ by $\End_{\GLd}^\mu(V^{\ot m})$ and the centralizer of the image of $\rho$ by $\End^\rho_{\mcS_m}(V^{\ot m})$.
The following result has appeared before frequently in the literature (for example \cite{grassl1998computing}) but we know of no place where a proof is written down.

\begin{Theorem}\label{thm:centralizer}
Given the described representations $\mu$ and $\rho$, then
\begin{enumerate}\itemsep=0pt
\item[$(a)$] $\End^\rho_{\mcS_m^n}(V^{\ot m})=\langle \GLd\rangle_\mu$.

\item[$(b)$] $\End^\mu_{\GLd}(V^{\ot m})=\langle \mcS_m^n\rangle_\rho$.
\end{enumerate}
\end{Theorem}
\begin{proof}
 Part (b) follows from part (a) by the double centralizer theorem. Now consider the isomorphism $\varphi\colon \End(V)^{\ot m}\cong\End(V^{\ot m})$ given by
\begin{gather*}\varphi\bigg(\bigotimes_{i=1}^n\ott_{j=1}^m{M_{ij}}\bigg)\bigg(\ott_{i=1}^n\ott_{j=1}^m{v_{ij}}\bigg)=\ott_{i=1}^n\ott_{j=1}^m{M_{ij}v_{ij}}.\end{gather*}

 We want to f\/ind those elements of $\End(V^{\ot m})$ which commute with $\mcS_m^n$. So let $\sigma=(\sigma_1,\dots,\sigma_n)$ $\in\mcS_m^n$ and consider
\begin{gather*}
\sigma\varphi\bigg(\ott_{i=1}^n\ott_{j=1}^m{M_{ij}}\bigg)\bigg(\sigma^{-1}\bigg(\ott_{i=1}^n\ott_{j=1}^m{v_{ij}}\bigg)\bigg)
=\sigma\bigg(\ott_{i=1}^n\ott_{j=1}^m{M_{ij}v_{i\sigma_i(j)}}\bigg)\\
\qquad{} =\ott_{i=1}^n\ott_{j=1}^m{M_{i\sigma^{-1}_i(j)}v_{ij}}=\varphi\bigg(\ott_{i=1}^n\ott_{j=1}^m{M_{i\sigma^{-1}_i(j)}}\bigg)\bigg(\ott_{i=1}^n\ott_{j=1}^m{v_{ij}}\bigg).
\end{gather*}

The map $\varphi$ induces an isomorphism from $\End^\rho_{\mcS_m^n}(V^{\ot m})$ to the subalgebra $\Sigma_{\bfd}$ of $\End(V)^{\ot m}$ that is $\mcS_m^n$ invariant under the induced action. We look at its decomposition as a $\mcS_m^n$ module. Since $\mcS_m^n$ acts trivially on it, every non-zero irreducible submodule will be one dimensional. Every irreducible representation of $\mcS_m^n$ is the tensor product of $n$ irreducible $\mcS_m$ modules. So we see that an irreducible $\mcS_m^n$ submodule of $\Sigma_{\bfd}$ is spanned by a vector $s_1\ot\cdots\ot s_n$ where each~$s_i$ is a symmetric tensor in $\End(V_i)^{\ot m}$ since it is invariant under $\mcS_m$.

So we see that $\Sigma_{\bfd}=\ott_{i=1}^n{\Sigma^i_m}$ where $\Sigma^i_m$ are the symmetric tensors of $\End(V_i)^{\ot m}$. However, it is known that $\Sigma^i_m$ is generated as an algebra by elements of the form $\ot_{i=1}^m{g_i}$ for $g_i\in\GL(V_i)$, i.e., $\Sigma^i_m=\langle \GL(V_i)\rangle_{\mu_i}$, where $\mu_i$ is the restriction to $\GL(V_i)\actson V_i^{\ot m}$. This fact is the classical case of the centralizer algebra of the general linear group~\cite{brauer1937algebras}.

So we get $\Sigma_{\bfd}=\ott_{i=1}^n{\langle\GL(V_i)\rangle_{\mu_i}}$. However, this algebra is clearly generated as an algebra by elements of the form $g_1^{\ot m}\ot\cdots\ot g_n^{\ot m}$ and so we get that $\Sigma_{\bfd}\cong\langle \GLd\rangle_\mu$. So we get the equality $\End^\rho_{\mcS_m^n}(V^{\ot m})=\langle \GLd\rangle_\mu$.
\end{proof}

We now def\/ine a set of multilinear polynomials that generalize the trace powers that appear in the classical setting.

\begin{Definition} \label{def:Trsgima}
For $\sigma=(\sigma_1,\dots,\sigma_n)\in\mcS_m^n$, let $\sigma_i=(r_1\cdots r_k)(s_1\cdots s_l)\cdots$ be a disjoint cycle decomposition.
For such a $\sigma\in\mcS_m^n$, def\/ine the {\em trace monomials} by $\tn{Tr}_\sigma=T_{\sigma_1}\cdots T_{\sigma_n}$ on $\End(V)^{\op m}$, where
\begin{gather*}T_{\sigma_i}\bigg(\ott_{j=1}^n{M_{j1}},\dots, \ott_{j=1}^n{M_{jm}}\bigg)=\tn{Tr}(M_{ir_1}\cdots M_{ir_k})\tn{Tr}(M_{is_1}\cdots M_{is_l})\cdots\end{gather*}
 and extend multilinearly.
\end{Definition}

\begin{Theorem}\label{thm:multend}
The multilinear invariants of $\End(V)^{\op m}$ under the adjoint action of $\GLd$ are generated by the $\Tr_\sigma$.
\end{Theorem}
\begin{proof}
Let $F$ denote the space of multilinear functions from $\End(V)^{\op m}\cong(V\ot V^*)^{\op m}\to k$. We caution that $F$ is \emph{not} the set of \emph{linear functions} from $\End(V)^{\op m}$ to $k$, but the set of functions $f(M_1,\dots,M_m)$ from $\End(V)^{\op m}$ to $k$ that is multilinear, i.e., linear in each of the $m$ arguments. We recall that the universal property of tensor products states that the set of functions from $V\op W$ to $k$ that are linear in both arguments is isomorphic to the space $(V\ot W)^*$. Extending this, we can identify $F$ with $[(V\ot V^*)^{\ot m}]^{*}$ by the universal property of tensor product. We note that there is an $\GLd$-equivariant isomorphism $\beta \colon [(V\ot V^*)^{\ot m}]^*\xrightarrow{\simeq}[V^{\ot m}\ot (V^{\ot m})^*]^*$ induced by rearranging the order of the tensor product in the obvious way and the canonical isomorphism $(V^*)^{\ot m}\xrightarrow{\simeq}(V^{\ot m})^*$. We also have an isomorphism of the spaces
\begin{gather*}\alpha\colon \ \End(V^{\ot m})\xrightarrow{\simeq}\big[V^{\ot m}\ot (V^{\ot m})^*\big]^*\end{gather*} given by $\alpha(A)(v\ot\phi)=\phi(Av)$ and extending linearly, which is $\GL(V^{\ot m})$-equivariant. Since $\GLd$ is a subgroup of $\GL(V^{\ot n})$, we get a $\GLd$-equivariant isomorphism $\End(V^{\ot m})\xrightarrow{\simeq} F$ by the map $\beta^{-1}\circ\alpha$. This induces an isomorphism{\samepage \begin{gather*}\End^\mu_{\GLd}\big(V^{\ot m}\big)\cong F^{\GLd},\end{gather*} where $F^{\GLd}$ are the $\GLd$-invariant multilinear functions.}

Since $V^{\ot m}\cong V_1^{\ot m}\ot\cdots\ot V_n^{\ot m}$, we can write $\alpha=\ott_{i=1}^n{\alpha_i}$ where $\alpha_i$ are the induced isomorphisms $\End(V_i^{\ot m})\xrightarrow{\simeq} [V_i^{\ot m}\ot (V_i^{\ot m})^{*}]^{*}$. Note that the following holds for the isomorphism~$\alpha_i$:
\begin{enumerate}\itemsep=0pt
\item[(a)] $\tn{Tr}(\alpha_i^{-1}(v\ot\varphi))=\varphi(v)$,
\item[(b)] $\alpha_i^{-1}(v_1\ot\varphi_1)\circ\alpha_i^{-1}(v_2\ot\varphi_2)=\alpha_i^{-1}(v_1\ot\varphi_1(v_2)\varphi_2)$.
\end{enumerate}
We explain these two equalities in more familiar terms. Equality (a) is the statement that $\Tr(vu^T)=u^Tv=\langle u,v\rangle$ for $u,v$ in some vector space $U$ and $\langle \cdot,\cdot\rangle$ the usual inner product. Equality~(b) is similar, stating that $(v_1u_1^T)(v_2u_2^T)=v_1(u_1^T v_2)u_2^T=\langle u_1,v_2\rangle(v_1u_2^T)$ for $u_1$, $u_2$, $v_1$, $v_2$ any vectors in some vector space~$U$.

Since $\End^\mu_{\GLd}(V^{\ot m})\cong F^{\GLd}$, by Theorem~\ref{thm:centralizer}, the images of $\sigma\in\mcS_m^n$ under $\alpha$ are the generators of $F^{\GLd}$. For $\sigma=(\sigma_1,\dots,\sigma_n)$, we have
\begin{gather*}
\alpha(\sigma)\bigg(\ott_{i=1}^n\ott_{j=1}^m{v_{ij}}\ot\ott_{i=1}^n\ott_{j=1}^m{\phi_{ij}}\bigg)=\bigg(\ott_{i=1}^n\ott_{j=1}^m{\phi_{ij}}\bigg) \bigg(\ott_{i=1}^n\ott_{j=1}^m{v_{i\sigma_i^{-1}(j)}}\bigg)\\
\qquad {} =\prod_{i=1}^n{\phi_{im}\big(v_{i\sigma_i^{-1}(m)}\big)}=T_{\sigma_1^{-1}}\cdots T_{\sigma_n^{-1}}=\Tr_{\sigma^{-1}},
\end{gather*}
where the f\/irst equality is a consequence of equality (a) and the second equality is a consequence of equality (b) above.
\end{proof}

Consider a vector of natural numbers $P=(p_1, \dots, p_{|P|})$ with elements from $[m]:=\{1,\dots, m\}$. We extend Def\/inition \ref{def:multihom} slightly.
\begin{Definition} \label{def:TrMsgima}
Given a vector $P=(p_1,\dots,p_{|P|})$ with all $p_i\in[m]$, and $\sigma\in\mcS^n_{|P|}$, def\/ine the polynomials on $\End(V)^{\op m}$ by their action on simple tensors in $\bigotimes_{i=1}^n \End(V_i)$,
\begin{gather*}\Tr^P_{\sigma}=\Tr_{\sigma}\bigg(\ott_{j=1}^n{M_{jp_1}},\dots,\ott_{j=1}^n{M_{jp_{|P|}}}\bigg)\end{gather*}
and extending multilinearly to $\End(V)^{\op m}$.
 \end{Definition}

Note that Def\/inition \ref{def:TrMsgima} dif\/fers from Def\/inition \ref{def:Trsgima} in that it allows for repetition of a matrix in the arguments. So we see that it is precisely a restitution of the multilinear invariants given in Def\/inition \ref{def:TrMsgima}. We now prove this formally.

\begin{Theorem}\label{thm:genend}
 The ring of $\GLd$-invariants of $\End(V)^{\op m}$ is generated by the $\tn{Tr}^P_\sigma$.
\end{Theorem}
\begin{proof}
We observed previously that the multihomogeneous invariants generate all the invariants. Let $W=\End(V)$. Consider a multihomogeneous invariant function of degree $\alpha=(\alpha_1,\dots,\alpha_m)$ (where some of the $\alpha_i$ might be zero) in $k[W^{\op m}]$. It is the restitution of a multilinear invariant in $k[W^{\op \alpha_1}\op\cdots\op W^{\op \alpha_m}]$. Let $|\alpha|=\sum\limits_{i=1}^m{\alpha_i}$.

By Proposition \ref{prop:rest}, we need only look at the restitutions of $\Tr_\sigma$, for $\sigma\in\mcS_{|\alpha|}^n$. What we get is the following:
\begin{gather}\label{eq1}\Tr_{\sigma}\Big(\underbracket{M_1,\dots,M_1}_{\text{$\alpha_1$}},\dots,\underbracket{M_m,\dots,M_m}_{\text{$\alpha_m$}}\Big).\end{gather}
We now def\/ine \begin{gather*}P=\Big(\underbracket{1,\dots,1}_{\text{$\alpha_1$}},\dots,\underbracket{m,\dots,m}_{\text{$\alpha_m$}}\Big)\end{gather*} and we see that $\Tr^P_{\sigma}$ is equal to the function in equation \eqref{eq1}.
\end{proof}

We can visualize the invariants $\Tr^P_\sigma$ in an intuitive way. For those familiar with tensor networks, they will recognize the following diagrams. For those unfamiliar, for this particular situation, the rules are very simple. Those interested in knowing more about these invariants as tensor networks can see \cite{biamonte2013tensor}.

We represent the matrix $M_i\in\End(V)^{\op m}$ by the following picture:

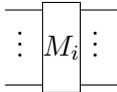
\begin{center}
\begin{tikzpicture}

 \draw (0,0) -- (.5,0);
 \draw (.2,-.35) node{$\vdots$};
 \draw (0,-1) -- (.5,-1);
 \draw (.5,-1.1) rectangle (1,.1);
 \draw (.75,-.5) node{$M_i$};
 \draw (1,0) -- (1.5,0);
 \draw (1.2,-.35) node{$\vdots$};
 \draw (1,-1) -- (1.5,-1);
\end{tikzpicture}
\end{center}
In the picture, there are $n$ wires on both sides of the box. Each wire represents one of the vector spaces in $V=\ott_{i=1}^n{V_i}$. The following picture describes how to represent the multiplication~$M_iM_j$:
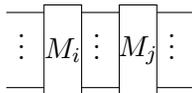
\begin{center}
\begin{tikzpicture}

 \draw (0,0) -- (.5,0);
 \draw (.2,-.35) node{$\vdots$};
 \draw (0,-1) -- (.5,-1);
 \draw (.5,-1.1) rectangle (1,.1);
 \draw (.75,-.5) node{$M_i$};
 \draw (1,0) -- (1.5,0);
 \draw (1.2,-.35) node{$\vdots$};
 \draw (1,-1) -- (1.5,-1);
 \draw (1.5,-1.1) rectangle (2,.1);
 \draw (1.75,-.5) node{$M_j$};
 \draw (2,0) -- (2.5,0);
 \draw (2.2,-.35) node{$\vdots$};
 \draw (2,-1) -- (2.5,-1);
\end{tikzpicture}
\end{center}

Given a matrix $M\in \End(V)$, we can take a partial trace relative to one of its subsystems. Suppose we trace out the subsystem $V_1$. In the diagram, this would look like the following:
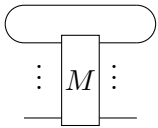
\begin{center}
\begin{tikzpicture}

 \draw (0,0) -- (.5,0);
 \draw (.2,-.35) node{$\vdots$};
 \draw (0,-1) -- (.5,-1);
 \draw (.5,-1.1) rectangle (1,.1);
 \draw (.75,-.5) node{$M$};
 \draw (1,0) -- (1.5,0);
 \draw (1.2,-.35) node{$\vdots$};
 \draw (1,-1) -- (1.5,-1);
 \draw (0,0) arc (270:90:.25);
 \draw (1.5,0) arc (-90:90:.25);
 \draw (0,.5) -- (1.5,.5);
\end{tikzpicture}
\end{center}

Every invariant can be built up by combining these two procedures in any way possible until there are no more ``hanging'' wires. The resulting picture is a series of loops aligned in $n$ rows. The loops are given by the disjoint cycle decomposition of some permutation and so each invariant is specif\/ied by some element in $\mcS_m^n$ as we saw before.

\begin{Example}\label{ex:inv}
We consider a specif\/ic invariant for $(M_1,M_2)\in \End(V_1\ot V_2)^{\op 2}$:
\begin{center}
\begin{tikzpicture}
\draw (-1.5,-.25) node{$\Tr^{(1,1,2)}_{(23),(12)}(M_1,M_2)=$};
 \draw (.25,0) arc (270:90:.25);
 \draw (.75,0) arc (-90:90:.25);
 \draw (.25,.5) -- (.75,.5);
 \draw (.25,-1) arc (270:90:.25);
 \draw (.25,-1) -- (2,-1);
 \draw (.25,-.6) rectangle (.75,.1);
 \draw (.5,-.25) node{$M_1$};
 \draw (.75,-.5) -- (1.5,-.5);
 \draw (1.5,-.6) rectangle (2,.1);
 \draw (1.5,0) arc (270:90:.25);
 \draw (1.5,.5) -- (3.25,.5);
 \draw (2,-1) arc (-90:90:.25);
 \draw (1.75,-.25) node{$M_1$};
 \draw (2,0) -- (2.75,0);
 \draw (2.75,-.6) rectangle (3.25,.1);
 \draw (3,-.25) node{$M_2$};
 \draw (2.75,-1) arc (270:90:.25);
 \draw (3.25,-1) arc (-90:90:.25);
 \draw (2.75,-1) -- (3.25,-1);
 \draw (3.25,0) arc (-90:90:.25);

\end{tikzpicture}
\end{center}

The disjoint cycle decomposition of the f\/irst permutation is $(1)(23)$ telling us that in the top row the f\/irst box receives a loop and the next two boxes receive a joint loop. Similarly in the bottom row, we see that $(12)(3)$ tells that the f\/irst two boxes receive a~joint loop and last box a~loop on its own. The vector $(1,1,2)$ tells us that the boxes are labeled $M_1$, $M_1$, and $M_2$ in that order.
\end{Example}

\subsection[Restrictions on the $\Tr_\sigma^P$]{Restrictions on the $\boldsymbol{\Tr_\sigma^P}$}

Much is known about the ring of invariants of $\End(V)^{\oplus m}$ under the adjoint representation of $\GL(V)$ including that it is Cohen--Macaulay and Gorenstein~\cite{hochster1974rings}; see Formanek~\cite{formanek1991polynomial} for an exposition.

The following theorem about generators of this invariant ring is classical \cite[Section 2.5]{kraft2000classical}.
\begin{Theorem}[\cite{kraft2000classical}]\label{thm:classicalgens}
The ring $k[\End(V)^{\op m}]^{\GL(V)}$ is generated by{\samepage \begin{gather*}\Tr(M_{i_1}\cdots M_{i_\ell}),\qquad 1\le i_1,\dots,i_\ell\le m,\end{gather*} where $\ell\le\dim (V)^2$. If $\dim (V)\le 3$, $\ell\le\binom{\dim(V)+1}{2}$ suffices~{\rm \cite{kraft2000classical,razmyslov1974trace}}.}

Furthermore, it is well known that $k[\End(V)]^{\GL(V)}$ is generated by the polynomials $\Tr(M^k)$ for $1\le k\le\dim(V)$ and that furthermore, these polynomials are algebraically independent $($cf.~{\rm \cite{kraft2000classical})}.
\end{Theorem}

Note that the degree of $\Tr^P_{\sigma}$ as a polynomial in the matrix entries equals $|P|$.
Theorem \ref{thm:classicalgens} does not provide a bound on the generating degree for the invariant ring of the {\em local} action $k[\End(V)^{\op m}]^{\GLd}$. The reason is that some trace monomials do not factorize into trace monomials of smaller degree, for example see Example \ref{ex:inv}. If it could, we could separate it as two separate invariants placed adjacent to each other.

It is an interesting question to know if one can determine when such an invariant can be factorized. Unfortunately, this problem is $\mathsf{NP}$-complete as we will show by reducing to the following problem. Suppose we are given $n$ multisets $S_1,\dots, S_n$. Def\/ine $\Sigma(S_i):=\sum_{j\in S_i}{j}$. Now suppose $\Sigma(S_i)=\Sigma(S_j)$ for all $i,j$. Then we want to know if every set admits a partition $S_j=A_j\sqcup B_j$ such that $\Sigma(A_j)=\Sigma(A_i)$ for all $i,j$ and likewise for the sets $B_i$. Deciding this problem is $\mathsf{NP}$-complete if $n>1$ \cite{turner2016subtilings}.

\begin{Proposition}\label{prop:hardfactor}
 For $n>1$, deciding if $\Tr^P_{\sigma}$ factorizes is $\mathsf{NP}$-complete.
\end{Proposition}
\begin{proof}
 The containment of this decision problem in $\mathsf{NP}$ is clear. We simply need to prove hardness. Suppose we could decide this problem, then we could decide it for $\Tr^P_{\sigma}(M)$, the case when $m=1$. Then def\/ine the set~$S_i$ to be the cycle lengths in the disjoint cycle decomposition in $\sigma_i$. We see that $\Sigma(S_i)=\Sigma(S_j)$ for all $i,j$. Furthermore, we see that $\Tr^P_{\sigma}(M)$ factors if and only if every set $S_i$ admits a partition $S_i=A_i\sqcup B_i$ such that $\Sigma(A_i)=\Sigma(A_j)$ for all $i$, $j$ and likewise for the sets $B_i$.
\end{proof}

Proposition \ref{prop:hardfactor} cautions us about the wisdom of trying to f\/ind minimal complete sets of invariants by simply enumerating them and checking to see if they are redundant. This approach will involve solving many instances of an $\mathsf{NP}$-complete problem. However, such an enumeration procedure was recently proposed in \cite{gour2013classification} for SLOCC invariants. We will see later, that such invariants for $n$-qubit systems are of the form $\Tr^P_\sigma$ where the inputs are matrices of restricted form.

Theorem \ref{thm:classicalgens} does allow us to restrict the functions $\Tr^P_\sigma$ that act as candidates for generators for the ring $k[\End(V)]^{\GLd}$ (Proposition \ref{prop:girth}).

\begin{Definition}
 The \emph{size} of $T^P_{\sigma_i}$ is def\/ined to be the size of the largest cycle in the disjoint cycle decomposition of $\sigma_i$.
\end{Definition}

\begin{Definition}
 Given a minimal set of generators, the \emph{girth} of $k[\End(V)^{\op m}]^{\GLd}$ is a tuple $(w_1,\dots,w_n)$ where $w_i$ is the maximum size of any $T^P_{\sigma_i}$ appearing in a generator. The girth of a~function $\Tr^{P}_\sigma$ is a~tuple $(s_1,\dots, s_n)$, where $s_i$ is the size of $T^P_{\sigma_i}$.
\end{Definition}

Note that the girth of the simple case $k[\End(V)]^{\GL(k^{d_i})}$ is simply the minimum $\ell$ such that the functions $\{\Tr(M_{i_1}\cdots M_{i_\ell})\colon 1\le i_1,\dots,i_\ell\le m\}$ generate it. We put a partial ordering on girth as follows: $(w_1,\dots, w_n)< (w'_1,\dots,w'_n)$ if there exists $i$ such that $w_i<w'_i$ and for no $j$ do we have $w'_j<w_j$. The girth is bounded locally by the square of the dimension.
\begin{Proposition}\label{prop:girth}
 If $(w_1,\dots,w_n)$ is the girth of $k[\End(V)^{\op m}]^{\GLd}$, then $w_i\le y_i$, where $y_i$ is the girth of $k[\End(V_i)^{\op m}]^{\GL(k^{d_i})}$. In particular for $V=V_1\ot\cdots\ot V_n$, the girth of $k[\End(V)^{\op m}]^{\GLd}$ is bounded by $(d_1^2,\dots,d_n^2)$. If $d_i\le 3$, then the girth is bounded by $\B(\binom{d_1+1}{2},\dots,\binom{d_n+1}{2}\B)$.
\end{Proposition}
\begin{proof}
 First note that $T^P_{\sigma_i}$ lies in the invariant ring $R_i=k[\End(V_i)^{\op m}]^{\GL(k^{d_i})}$. Thus it has size at most $y_i$, where $y_i$ is the girth of $R_i$. Now apply Theorem \ref{thm:classicalgens}.
\end{proof}

\section{Closed orbits}\label{sec:closedorbits}

We f\/irst give an a suf\/f\/icient condition for $(M_1,\dots,M_m)\in\End(V^{\op m})$ to have a closed $\GLd$ orbit, where $V$ is a Hilbert space throughout this section. We show that, in particular, tuples of normal matrices over $\C$ satisfy the given properties. Since density operators are Hermitian, they are immediately normal.

\begin{Theorem}[\cite{MR1304906}]\label{thm:separatedness}
 Given a reductive group acting rationally on vector space, for two distinct closed orbits, there is a polynomial invariant that takes different values on each.
\end{Theorem}

So we seek to show that normal matrices have closed orbits. This will show that polynomial invariants serve as a complete set of invariants when restricted to density operators. As we noted before, the Zariski closures and Euclidean closures of orbits coincide for reductive groups acting rationally. As such, Theorem \ref{thm:separatedness} implies that two closed orbits are distinguishable by continuous invariants if and only if they are distinguishable by polynomial invariants. Returning to Remark \ref{rem:caveat}, this implies that we need not consider the more general notion of polynomial invariants as often def\/ined in the literature in order to f\/ind a complete set of invariants.

\begin{Definition}
A decomposition $V=W\op W^{\perp}$, $W,W^\perp\ne\{0\}$, is said to be \emph{separable} if there exists a cocharacter of $\GLd$, $\lambda(t)$ such that $\forall\, w\in W$, $\lim\limits_{t\to 0}{\lambda(t)w}=0$, and $\forall\, w\in W^{\perp}$, $w\ne0$, $\lim\limits_{t\to 0}{\lambda(t)w}\ne0$. We call $\lambda(t)$ a \emph{separating subgroup} of the decomposition (this group is not unique).
\end{Definition}

\textbf{Caveat:} The def\/inition of a separable decomposition depends on the order in which the summands are written. If $V=W\op W^{\perp}$ is a separable decomposition, it is not necessarily the case that $W^\perp\op W$ is also a separable decomposition.

Given an arbitrary cocharacter of $\GLd$, it is not clear that there is necessarily a separable decomposition that one can associate to it. The following lemma allows us to replace a cocharacter by one that does have a separable decomposition associated to it that does not af\/fect limits.

\begin{Lemma}\label{lem:cochar}
 Let $\lambda(t)$ be a cocharacter of $\GLd$. Then there exists another cocharacter $\mu(t)$ such that the following assertions hold:
 \begin{enumerate}\itemsep=0pt
 \item[$(a)$] $\lim\limits_{t\to0}{\lambda(t)M\lambda(t)^{-1}}=\lim\limits_{t\to 0}{\mu(t)M\mu(t)^{-1}}$ for all $M\in\End(V)$ such that the limit exists,
 \item[$(b)$] $\mu(0):=\lim\limits_{t\to0}{\mu(t)}$ exists,
 \item[$(c)$] unless $\lambda(t)=t^{\alpha}\id$, then $\mu(0)$ has two nontrivial eigenspaces with eigenvalues $0$, $1$.
 \end{enumerate}
\end{Lemma}
\begin{proof} We can diagonalize $\lambda(t)$ by some element $g\in\GLd$. Thus it suf\/f\/ices to prove the aboves statements for diagonal cocharacters. If $\lambda(t)$ is a diagonal cocharacter, the diagonal entries are of the form $t^{\alpha_i}$, $\alpha_i\in\mathbb{Z}$ (cf.~\cite{kraft2000classical}). Let $\alpha_m$ be the most negative exponent, or if all $\alpha_i$ are strictly positive, then let $\alpha_m$ be the smallest positive exponent. Then let $\mu(t)=t^{-\alpha_m}\lambda(t)$. We see that for any $M\in\End(V)$, $\lambda(t)M\lambda(t)=\mu(t)M\mu(t)^{-1}$. Therefore $\lim\limits_{t\to0}{\lambda(t)M\lambda(t)^{-1}}=\lim\limits_{t\to0}{\mu(t)M\mu(t)^{-1}}$ whenever the limit exists.

 Furthermore, we see that $\mu(t)$ has diagonal entries all non-negative powers of $t$. Therefore, $\lim\limits_{t\to 0}{\mu(t)}$ exists and is in fact equal to $\mu(0)$. Furthermore, unless $\mu(t)=t^{\alpha}\id$, $\mu(0)$ will have both zeros and ones on the diagonal. Thus it will have to non-trivial eigenspaces with eigenva\-lues~$0$,~$1$.
\end{proof}

We now show how to construct separable decompositions as it is not clear that they necessarily exist. We must use cocharacters of the form as in Lemma~\ref{lem:cochar}.

\begin{Lemma}\label{lem:sepdecomp}
 Given a cocharacter as in Lemma~{\rm \ref{lem:cochar}}, except for $\lambda(t)=t^\alpha\id$, we can associate it to a separable decomposition for which it is the separating subgroup.
\end{Lemma}
\begin{proof}
 Let $\mu(t)$ be a cocharacter as in Lemma~\ref{lem:cochar}. Then we know that $\mu(0):=\lim\limits_{t\to 0}{\mu(t)}$ exists and is a matrix. Then $\mu(0)$ has two eigenspaces, one attached to eigenvalue 1 and the other to eigenvalue 0. Let $W$ be the null space of $\mu(0)$. Then consider the decomposition $V=W\op W^{\perp}$. Then $\forall\, w\in W$, $\lim\limits_{t\to0}{\mu(t)W}=\mu(0)W= 0$, and $\forall\, w\in W^{\perp}$ then $\lim\limits_{t\to0}{\mu(t)w}=\mu(0)w$, which projects~$W^{\perp}$ onto the eigenspace attached to the eigenvalue 1. This means that the only $v\in W^{\perp}$ such that $\mu(0)v=0$ is $v=0$. So this a separable decomposition for which~$\mu(t)$ is the separating subgroup.
\end{proof}

Let us analyze which decompositions are separable. Let us f\/irst analyze the case that $\lambda(t)=\ott_{i=1}^n{\lambda_i(t)}$ is as in Lemma \ref{lem:cochar} and is diagonal. Then $\lambda_i(t)$ is diagonal and can be taken to have diagonal entries with all non-negative powers of $t$. Thus, for every $i$, we can decompose $V_i=W_i\op W_i^{\perp}$ where $\lim\limits_{t\to 0}{\lambda(t)w}=0$ for all $w\in W_i$ and $\lambda(t)w=w$ for all $w\in W_i^{\perp}$. Then $(W_1^\perp \ot\cdots\ot W_n^\perp)^\perp$ gets sent to zero by $\lambda(t)$.
It is easy to see that every separable decomposition for a diagonal cocharacter is of the form \begin{gather*}\big(W_1^\perp \ot\cdots\ot W_n^\perp\big)^\perp\op \big(W_1^\perp \ot\cdots\ot W_n^\perp\big).\end{gather*} From here, it is easy to see that every separable decomposition is of the same form by taking the $\GLd$ orbits of diagonal cocharacters.

Given a matrix $M\in\End(V)$, we are interested in separable decompositions $W\op W^{\perp}$ such that $M(W)\subseteq W$. Let $P_W$ and $P_{W^{\perp}}$ be the projection operators onto each of the two subspaces. Then def\/ine $M|_W:=P_W(M)$ and $M|_{W^\perp}:=P_{W^{\perp}}(M)$.

\begin{Proposition}\label{prop:closure}
 For every separable decomposition $V=W\op W^\perp$ such that $M(W)\subseteq W$, $M|_W\op M|_{W^{\perp}}$ is in the orbit closure of $M$.
\end{Proposition}
\begin{proof}
 We can write $M$ as
 \begin{gather*}
 M=\bordermatrix{ & W & W^\perp\cr
W &A&B\cr
W^\perp &0&C}.
\end{gather*}
We know that $W=(W_1^\perp \ot\cdots\ot W_n^\perp)^\perp$ for subspaces $W_i\subseteq V_i$. Then we let $\lambda(t)=\ott_{i=1}^m{\lambda_i(t)}$ where
\begin{gather*}\lambda_i(t)=
 \bordermatrix{ & W_i & W_i^{\perp}\cr
W_i &tI&0\cr
W_i^{\perp} &0&I}.\end{gather*} Then we see that
\begin{gather*}\lambda(t)=
 \bordermatrix{ & W & W^{\perp}\cr
W &tQ(t)&0\cr
W^{\perp} &0&I}, \end{gather*} where $Q(t)$ is a diagonal matrix with non-zero entries being non-negative powers of $t$. In particular, it is invertible. Then we have that
\begin{gather*}
 \bordermatrix{ & W & W^{\perp}\cr
W &tQ(t)&0\cr
W^{\perp} &0&I}\cdot
 \bordermatrix{ & W & W^{\perp}\cr
W &A&B\cr
W^{\perp} &0&C}\cdot
\bordermatrix{ & W & W^{\perp}\cr
W &t^{-1}Q(t^{-1})&0\cr
W^{\perp} &0&I}
\\
\qquad{} =
\bordermatrix{ & W & W^{\perp}\cr
W &A&tQ(t)B\cr
W^{\perp} &0&C},
\end{gather*}
which we see takes $M\to M|_W\op M|_{W^{\perp}}$ as $t\to0$.
\end{proof}

\begin{Theorem}\label{thm:splitexact}
 A matrix $M$ has a closed $\GLd$ orbit if there exists some $M'\in\GLd.M$ such that for every separable decomposition $V=W\op W^{\perp}$ satisfying $M'(W)\subseteq W$, then $M'(W^{\perp})\subseteq W^{\perp}$.
\end{Theorem}
\begin{proof}
 Suppose that $M$ does not have a closed orbit, so it can be written as $M=M_s+M_n$ where~$M_s$ has a closed orbit and $M_n$ is in the null cone. Then by Theorem~\ref{thm:hbcriterion}, there is a~cocharac\-ter~$\lambda(t)$ taking~$M\to M_s$. We can assume that $\lambda(t)$ satisf\/ies the properties of Lemma~\ref{lem:cochar}. Let\-ting~$W$ be the kernel of $\lambda(0)$, we see that $V=W\op W^{\perp}$ is a separable decomposition.

 Let $w\in W$. We note that $\lambda(t)Mw=\lambda(t)M\lambda(t)^{-1}\lambda(t)w$. We know that $\lambda(t)M\lambda(t)^{-1}$ is a~mat\-rix in which only non-negative powers of $t$ appears. Furthermore, every entry of $\lambda(t)w$ is scaled by some positive power of $t$. Therefore every element of $\lambda(t)Mw$ is scaled by a positive power of $t$, so $\lim\limits_{t\to0}{\lambda(t)Mw}=0$. Therefore $M(W)\subseteq W$.

 Notice that a similar argument shows that $M_s(W)\subseteq W$ and therefore we can write
 \begin{gather*}
 M_s=\bordermatrix{ & W & W^{\perp}\cr
W &A&B\cr
W^{\perp} &0&C}.
\end{gather*} However, by Proposition \ref{prop:closure}, we can assume that $B=0$. That is to say, $M_s(W^{\perp})\subseteq M_s(W^{\perp})$.

 If $u\in W^{\perp}$, then $\lim\limits_{t\to0}{\lambda(t)u}$ lies in the eigenspace of $\lambda(0)$ attached to the eigenvalue of $1$ (it may not be the case that this eigenspace is orthogonal to the kernel of $\lambda(0)$). However, we note that $\lambda(t)M_n\lambda(t)^{-1}$ has every entry scaled by a positive power of $t$, and thus $\lambda(t)M\lambda(t)^{-1}\lambda(t)u$ has all entries scaled by some positive power of $t$ and thus $\lim\limits_{t\to 0}{\lambda(t)M_nu}=0$. This implies that~$M_nu$ is in $W$ and therefore, and since $M_s(u)\in W^{\perp}$, $W^{\perp}$ is not an invariant subspace.
\end{proof}

We can show that matrices that respect orthogonal decompositions have closed orbits. The prime example are normal matrices as these are precisely the matrices with an orthogonal basis by the spectral theorem.

\begin{Theorem}\label{thm:normal}
 For $\GLd\actson\End(V)^{\op m}$, tuples of normal matrices have closed orbits.
\end{Theorem}
\begin{proof}
It suf\/f\/ices to show that for $\GLd\actson\End(V)$, matrices with an orthogonal eigenbasis have closed orbits. Then the result follows from the fact that, if such a $(M_1,\dots,M_m)$ acted on by $\GLd$ did not have a closed orbit, then projecting onto some coordinate, say $i$, would induce a non-trivial limit point, implying that the matrix $M_i$ did not have a closed orbit.

 Let $M$ have an orthogonal eigenbasis. Then let $V=W\op W^{\perp}$ be a separable decomposition such that $M(W)\subseteq W$. It must be that $W$ is a direct sum of eigenspaces of $M$ (here, by eigenspace, we mean any subspace which $M$ acts on by scaling). Since the eigenspaces of $M$ are orthogonal (in the sense that given two vectors in two dif\/ferent eigenspaces, they are orthogonal), we immediately have that $W^{\perp}$ is a direct sum of eigenspaces. Thus $W^{\perp}$ is an invariant subspace of $M$. Then applying Theorem \ref{thm:splitexact}, we get that $M$ has a closed orbit.
\end{proof}

\begin{Corollary}\label{cor:densitysep}
The $\GLd$ orbits of tuples of density matrices are closed, so they can be separated by polynomial invariants. Moreover, two Hermitian matrices are in the same $\GLd$ orbit if and only if they are in the same $\Ud$ orbit.
\end{Corollary}
\begin{proof}
 We know from Proposition \ref{prop:glimpliesu} that two density operators are in the same $\GLd$ orbit if and only if they are in the same $\Ud$ orbit. We know from Theorem \ref{thm:normal} that tuples of density operators have closed orbits. We know from Theorem \ref{thm:separatedness} that two closed orbits can be distinguished by invariants if and only if they are distinct.
\end{proof}

\begin{Corollary}\label{cor:complete}
 The functions $\Tr^P_\sigma$ form a complete set of invariants for tuples of density ope\-rators under the action of~$\Ud$.
\end{Corollary}
\begin{proof}
 This follows from Corollary \ref{cor:densitysep} and Theorem~\ref{thm:genend}.
\end{proof}

So we know that two tuples of density operators are not in the same $\Ud$ orbit if and only if there is some $\Tr^P_\sigma$ on which they take dif\/ferent values. We know from Theorem \ref{thm:noetherian}, that there exists a f\/inite set of functions $\Tr^P_\sigma$ that forms a complete system of invariants. This theorem does not tell us what such a f\/inite set may be. However, we have a bound given by the following result.

\begin{Theorem}[\cite{derksen2015computational}]\label{thm:genbound}
 Let $\rho\colon G\to \GL(V)$ be a reductive group acting rationally. Let $f_1,\dots,f_\ell$ be homogeneous invariants, with maximum degree $\gamma$, such that their vanishing locus is $\mcN_V$. Then \begin{gather*}\beta_G(V)\le\max \left\{2,\frac{3}{8}\dim\big(k[V]^G\big)\gamma^2\right\}.\end{gather*} Furthermore, $\gamma$ is bounded by $CA^m$ where $C$ is the degree of $G$ as a variety and $m=\dim(\rho(G))$. Since $\rho$ is a rational map, it can be viewed as a vector valued function with a rational function in each coordinate. Then $A$ is defined to be the maximum degree of any of these coordinate rational functions.
\end{Theorem}

As we noted earlier, $\GLd$ can be replaced by $\SLd:=\times_{i=1}^n{\SL(V_i)}$ since this group action has the same invariant ring.

\begin{Corollary}\label{cor:maincor}
 The polynomials $\Tr^P_\sigma$ of girth at most $(d_1^2,\dots, d_n^2)$ and degree at most
 \begin{gather*}
\max \left\{2,\frac{3}{8}\max\{d_i\}m^2\dim(V)^4(2n)^{2\delta}\right\},
 \end{gather*}
 where $\delta=\sum\limits_{i=1}^n{(d_i-1)}$, give a finite complete set of invariants for LU-equivalence of $m$-tuples of density operators.
\end{Corollary}
\begin{proof}
 The f\/irst part of the statement follows from Proposition \ref{prop:girth}. The degree bound comes from Theorem \ref{thm:genbound} and the following facts. $\SLd$ is def\/ined by equations of degrees $d_i$ since~$\SLd$ consists of tuples of matrices each of determinant one, so $C\le \max{d_i}$. It is easy to see that $A=2n$ as taking the Kronecker product of~$n$ matrices gives monomials of degree $n$ in the entries of the original matrices and conjugation is a quadratic action. Since the representation of~$\SLd$ is faithful $\dim(\rho(\SLd))=\dim(\SLd)=\sum\limits_{i=1}^n{(d_i-1)}$. Lastly, we note that $\dim(k[V]^G)\le \dim(k[V])=\dim(V)$ for any $G\actson V$.
\end{proof}

\section{SLOCC invariants for any number of qubits}\label{sec:slocc}
We now wish to relate the invariants of $\Sl{2}:=\times_{i=1}^n{\SL(\C^2)}$ by left multiplication on $V^{\op m}$, where $V=(\C^2)^{\ot n}$, to the invariants of $\Sl{2}$ by conjugation on $\End(V)^{\op m}$. The relevant property we use is that the action of $\Sl{2}$ on $V^{\op m}$ is \emph{self-dual}. This means that the standard action of $\Sl{2}$ on $\C^2$ is isomorphic to the representation of $\Sl{2}$ on $(\C^2)^*$ given by $g.\varphi=\varphi(g^{-1})$. To state this more formally:

\begin{Definition}
 A representation $\rho\colon G\to \GL(V)$ is called self-dual if $\rho\simeq\rho^*$, where $\rho^*$ is the induced contragradient representation on $V^*$.
\end{Definition}

The action of $\SL(\C^2)$ on $\C^2$ by left multiplication is self-dual. Let $T=\begin{pmatrix}0&1\\-1&0\end{pmatrix}$. Then for any $g\in \SL(\C^2)$, $TgT^{-1}=(g^{-1})^T$. We consider the map $\phi\colon \C^2\to (\C^2)^*$ given by $\phi(v)=(Tv)^T$. Then \begin{gather*}\phi(gv)=(Tgv)^T=\big(TgT^{-1}Tv\big)^T=(Tv)^Tg^{-1}.\end{gather*} This gives an equivariant isomorphism between the standard action of $\SL(\C^2)$ and its induced contragradient representation.

\begin{Lemma}\label{lem:sl2selfdual}
The action of $\rho\colon \Sl{2}\to \GL(V^{\op m})$ by left multiplication is self-dual.
\end{Lemma}
\begin{proof}
 Let $\phi\colon V^{\op m}\to (V^*)^{\op m}$ be the linear map given by $\phi(\op_{i=1}^m{v_i})=\bop_{i=1}^m{(T^{\ot n}v_i)^T}$. Let $g=\ot_{i=1}^n{g_i}\in \rho(\Sl{2})$. Then
 \begin{gather*}
 \phi\big(g\op_{i=1}^m{v_i}\big)=\bop_{i=1}^m{\big(T^{\ot n}gv_i\big)^T}=\bop_{i=1}^m{\big(T^{\ot n}g\big(T^{-1}\big)^{\ot n}T^{\ot n}v_i\big)^T}\\
 \hphantom{\phi\big(g\op_{i=1}^m{v_i}\big)}{}
 =\bop_{i=1}^m{\big(T^{\ot n}v_i\big)^T\big({\ot}_{i=1}^n{Tg_iT^{-1}}\big)^T}\\
 \hphantom{\phi\big(g\op_{i=1}^m{v_i}\big)}{}
=\bop_{i=1}^m{\big(T^{\ot n}v_i\big)^T\big({\ot}_{i=1}^n{\big(g_i^{-1}\big)^T}\big)^T}=\bop_{i=1}^m{\big(T^{\ot n}v_i\big)^Tg^{-1}}.\tag*{\qed}
 \end{gather*}\renewcommand{\qed}{}
\end{proof}

Let $G\actson V$ be a self-dual representation, given by $\rho$. Then there is an isomorphism $\phi\colon \rho\to\rho^*$. Since it is a linear map, there is a matrix $S$ such that $\phi(v)=(Sv)^T$. Then \begin{gather*}\phi(\rho(g)v)=(S\rho(g)v)^T=(Sv)^T\big(S\rho(g)S^{-1}\big)^T=(Sv)^Tg^{-1}.\end{gather*} Thus we have that a representation $\rho$ is self-dual if and only if there exists a matrix $S$ such that $S\rho(g)S^{-1}=\rho(g^{-1})^T$ for all $g\in G$.

Suppose the representation $\rho\colon G\to \GL(V)$ on $V$ is self-dual. Let $\phi\colon \rho\to\rho^*$ be the equivariant isomorphism. This induces an action on $V^{\op m}$, which is clearly self-dual. Then there is an equivariant inclusion of $\psi\colon V^{\op m}\hookrightarrow (V\op V^*)^{\op m}$ given by
\begin{gather*}
\bop_{i=1}^m{v_i}\mapsto \bop_{i=1}^m{(v_i,\phi(v_i))},\\
g.\bop_{i=1}^m{(v_i,\phi(v_i))}=\bop_{i=1}^m{(\rho(g)v_i,\rho^*(g)\phi(v_i))}.
\end{gather*}

So let us consider the invariants on $(V\op V^*)^{\op m}$ with the above action. We f\/irst look at the multilinear invariants; from these we can construct all invariants. Let $I$ be the ideal def\/ining the image of $V\op V^*$ inside of $\End(V)$ under the Segre embedding. Recall that the Segre embedding of $V\op W$ is the map $(v,w)\mapsto v\ot w$. Also recall that the ideal def\/ining a variety is the set of polynomials that vanish identically on the variety. The image of the Segre embedding is $G$-stable and so its ideal is also $G$-stable.

\begin{Proposition}[\cite{nagata1963invariants}]\label{prop:quotient}
 Let $G$ act on a subvariety $X\subseteq V$. If $G$ is reductive, and its ideal, $I\subseteq k[V]$, is a $G$-stable ideal, then $k[V]^G/(I\cap k[V]^G)\cong (k[V]/I)^G$.
 \end{Proposition}

\begin{Lemma}\label{lem:multinvs}
 $\C[(V\op V^*)^{\op m}]^G\cong\C[\End(V)^{\op m}]^G/(I\cap \C[\End(V)^{\op m}]^G)$.
\end{Lemma}
\begin{proof}
 The multilinear invariants are elements of $\End(V)^{\op m}$ of degree $d$ are elements of the space $(\End(V)^{\ot d})^*$ by the universal property of tensor product.The multilinear invariants of $(V\op V^*)$ of degree $d$, are also elements of $(\End(V)^{\ot d})^*$, lying in the image of the Segre embedding $V\op V^*\hookrightarrow\End(V)$. Furthermore, notice that the action of $G$ on $(V\op V^*)^{\op d}$ and on $\End(V)^{\op d}$ both turn into the action on $\End(V)^{\ot m}$ given by
 \begin{gather*}g.\ott_{i=1}^d{M_i}=\ott_{i=1}^d{\rho(g)M_i\rho(g)^{-1}}.\end{gather*}
 So the multilinear invariants are the same and by Proposition \ref{prop:rest}, the restitutions are the same. Proposition \ref{prop:quotient} f\/inishes the proof.
\end{proof}

Of course, we are not interested in the entire space $(V\op V^*)^{\op m}$ but rather the subset def\/ined by the image of $\phi\colon V^{\op m}\hookrightarrow (V\op V^*)^{\op m}$. This is also a $G$-invariant variety.

Let $\tilde{\phi}\colon V^{\op m}\to \End(V)^{\op m}$ be the map given by $\bop_{i=1}^m{v_i}\mapsto \bop_{i=1}^m{(v_i\ot v_i^T)S^T}.$
For the case that $m=1$, the image of $V\in \End(V)$ is matrices of the form $v\ot (v^T S^T)$, which is isomorphic to the Veronese variety of matrices of the form $v\ot v^T$. Thus the image of $V^{\op m}\in \End(V)^{\op m}$ is isomorphic to a direct sum of these Veronese varieties.

Now consider its ideal $I\subset \C[\End(V)^{\op m}]$. The action of $G$ on $\End(V)^{\op m}$ induces an action on the coordinate ring. As $I$ def\/ines an $G$-invariant variety, it is clear that $I$ is a $G$-stable ideal.

\begin{Theorem}\label{thm:selfdualinvs}
 Suppose $\rho\colon G\to \GL(V)$ acting on $V^{\op m}$ is self-dual and reductive. Let $I$ be the ideal of $\Im(\tilde{\phi})$. Then \begin{gather*}\C\big[V^{\op m}\big]^G\cong \C\big[\End(V)^{\op m}\big]^G/\big(I\cap \C\big[\End(V)^{\op m}\big]^G\big).\end{gather*}
\end{Theorem}
\begin{proof}
By Lemma \ref{lem:multinvs}, $\C[(V\op V^*)^{\op m}]^G\cong\C[\End(V)^{\op m}]^G/(I\cap \C[\End(V)^{\op m}]^G)$. The inva\-riants of $\C[\End(V)^{\op m}]$ are interpreted as invariants of $V^{\op m}$ by precomposition with $\tilde{\phi}$. Then the result follows from follows from Proposition \ref{prop:quotient}.
\end{proof}

We know that $\Sl{2}$ is self-dual by Lemma \ref{lem:sl2selfdual}. Unfortunately, $\SL(\C^n)$ is self-dual only when $n=2$. So this method only works for the group $\Sl{2}$. We relate this to the invariant ring $\C[\End(V)^{\op m}]^{\Sl{2}}$, which we have already described.

For the case $\Sl{2}$, $\tilde{\phi}\colon V\to\End(V)$ is given by $\tilde{\phi}(v)=v\ot v^T(T^{\ot n})^T$ which extends naturally to a map $\tilde{\phi}\colon V^{\op}\to\End(V)^{\op m}$. Then we def\/ine \begin{gather*}\tilde{\Tr}^P_\sigma(v_{m_1},\dots,v_{m_\ell}):=\Tr^M_\sigma\big(\tilde{\phi}(v_{m_1}),\dots,\tilde{\phi}(v_{m_\ell})\big).\end{gather*} This turns the polynomials $\Tr^P_\sigma$ into polynomials in $\C[V^{\op m}]$. These polynomials generate the ring of invariants. However, we haven't accounted for the relations introduced among them from restricting the variety def\/ined by the image of $\tilde{\phi}$, so many of these polynomials will be redundant.

\begin{Theorem}\label{thm:slocc}
 The functions $\tilde{\Tr}^P_\sigma$ of degree at most \begin{gather*}\max\left\{2,\frac{3}{2}m^2\dim(V)^2(n)^{6n}\right\}\end{gather*} generate the invariants for $\C[V^{\op m}]^{\Sl{2}}$ on $n$ qubits.
\end{Theorem}
\begin{proof}
 By Lemma \ref{lem:sl2selfdual}, the action of $\Sl{2}$ on $V$ by left multiplication is self-dual and reductive. Then by Theorem \ref{thm:selfdualinvs}, the generators of $\C[\End(V)]^{\Sl{2}}$ applied to the image of $\tilde{\phi}$ gives a~ge\-nerating set for $\C[V]^{\Sl{2}}$. The bound comes from applying Theorem \ref{thm:genbound}. The degree of $\Sl{2}$ is at most two as it is def\/ined by determinants of $2\times 2$ matrices. $\dim(\C[V^{\op m}]^{\Sl{2}})\le \dim(V)$, $A$ is $n$ as we are taking a Kronecker product of $n$ matrices, and $\Sl{2}$ has dimension $3n$.
\end{proof}

While Theorem~\ref{thm:slocc} gives a complete accounting of all the polynomial SLOCC invariants for an $n$ qubit system, as well as a f\/inite generating set of the ring, further work is necessary. The most obvious problem is that the degree bound is obtained by appealing to a general degree bound for reductive group actions. There is no reason to expect that it is optimal; indeed, we conjecture that a degree bound exists that is polynomial in the dimension of $V$. For small~$n$, explicit generating sets are known and the following table compares these degree bounds to the ones given by Theorem~\ref{thm:slocc}, for $m=1$.
\begin{gather*}
\begin{tabular}{ l | c | c}
 $n$ & \tn{Known minimal degree bounds} & \tn{Degree bound from Theorem }\ref{thm:slocc}\\
 \hline
 1 & 0 (trivial) & 6\\
 2 & 2 (classical) & $24\cdot2^{12}$\\
 3 & 4 \cite{parfenov2001orbits} & $96\cdot 3^{18}$\\
 4 & 6 \cite{MR2039690} & $384\cdot4^{24}$
\end{tabular}
\end{gather*}

We see that the above degree bound is very far of\/f. While one might be tempted to algorithmically f\/ind minimal sets of invariants by enumerating all invariants, the above bound does not give an indication of how long such a enumeration would take. The known minimal degree bounds have been found by a variety of methods. However, as the number of qubits grows, the general approach has been an analysis of the Hilbert series of the rings to determine degrees of generators along with the computations of covariants. For 5 qubits, this method is already computationally prohibitive. As such, if any progress is to made in this direction, a better theoretical understanding of these invariants is necessary rather than relying on computation.

The second issue is that the above invariants might not all be necessary. Indeed, for the case of four qubits, this turned out to be the case~\cite{verstraete2002four}, although this case was special as there were a~f\/inite number of normal forms describing all of the orbits. A classif\/ication in terms of geometric properties was later carried out for four qubits \cite{holweck2016entanglement}. This is not likely to be the case as the number of qubits grows. Nevertheless, there may be relations (although necessarily non-algebraic) among the invariants as a result of restricting to quantum states.

\subsection*{Acknowledgements}
The authors would like to acknowledge the helpful comments of the reviewers which greatly improved and strengthened this paper. J.~Turner would like to thank Llu\'is Vena for helpful discussions. The research leading to these results has received funding from the European Research Council under the European Union's Seventh Framework Programme (FP7/2007-2013)~/ ERC grant agreement No~339109.

\pdfbookmark[1]{References}{ref}
\LastPageEnding

\end{document}